\documentclass[reqno,11pt]{amsart}
\usepackage{amssymb}
\usepackage{amsmath}
\usepackage{amsfonts}
\usepackage{microtype}
\usepackage[canadian]{babel}
\usepackage{xcolor}
\usepackage{geometry}

\definecolor{myurlcolor}{rgb}{0,0,0.4}
\definecolor{mycitecolor}{rgb}{0,0.5,0}
\definecolor{myrefcolor}{rgb}{0.5,0,0}
\usepackage[pagebackref,draft=false]{hyperref}
\hypersetup{colorlinks,
linkcolor=myrefcolor,
citecolor=mycitecolor,
urlcolor=myurlcolor}

\usepackage[capitalize]{cleveref}

\newcommand{\beq}{\begin{equation}}
\newcommand{\eeq}{\end{equation}}
\newcommand{\N}{\mathbb{N}}
\newcommand{\Nplus}{\mathbb{N}_{> 0}}			% *strictly* positive natural numbers

\newcommand{\Q}{\mathbb{Q}}

\newcommand{\R}{\mathbb{R}}
\newcommand{\C}{\mathbb{C}}

\newcommand{\eps}{\varepsilon}

\newcommand{\rk}[1]{\mathrm{rk}(#1)}			% rank
\newcommand{\im}[1]{\mathrm{im}(#1)}
\newcommand{\sr}[1]{\mathsf{#1}}	% general font for semirings

\newcommand{\Graph}{\sr{Graph}}
\newcommand{\FGraph}{\sr{FinGraph}}

\newcommand{\frachom}{\rightsquigarrow} % fractional homomorphism
\newcommand{\graphinv}[1]{\eta_{#1}} % graph invariant associated to a linear-like semiring family
\newcommand{\graphinvfrac}[1]{\eta^{\mathfrak{frac}}_{#1}} % corresponding fractional graph invariant
\newcommand{\graphinvasymp}[1]{\eta^{\infty}_{#1}} % corresponding fractional graph invariant

\theoremstyle{plain}
\newtheorem{thm}{Theorem}[section]\Crefname{thm}{Theorem}{Theorems}
\newtheorem{lem}[thm]{Lemma}\Crefname{lem}{Lemma}{Lemmas}
\newtheorem{prop}[thm]{Proposition}\Crefname{prop}{Proposition}{Propositions}
\newtheorem{cor}[thm]{Corollary}\Crefname{cor}{Corollary}{Corollaries}
\Crefname{conj}{Conjecture}{Conjectures}
\Crefname{qstn}{Question}{Questions}
\newtheorem{defn}[thm]{Definition}\Crefname{defn}{Definition}{Definitions}
\newtheorem{prob}[thm]{Problem}\Crefname{prob}{Problem}{Problems}
\theoremstyle{remark}
\newtheorem{ex}[thm]{Example}\Crefname{ex}{Example}{Examples}
\newtheorem{rem}[thm]{Remark}\Crefname{rem}{Remark}{Remarks}
\Crefname{note}{Note}{Notes}
\numberwithin{equation}{section}

\allowdisplaybreaks

\let\originalleft\left
\let\originalright\right
\renewcommand{\left}{\mathopen{}\mathclose\bgroup\originalleft}
\renewcommand{\right}{\aftergroup\egroup\originalright}

\protected\def\verythinspace{%
	\ifmmode
		\mskip0.5\thinmuskip
	\else
		\ifhmode
			\kern0.08334em
		\fi
	\fi}
\renewcommand{\,}{\verythinspace}

\usepackage{todonotes}

\begin{document}

% vertical spacing in multiline equations
\setlength{\jot}{6pt}

%-------------------------------------------------------------------

%%%%%%%%%%%% title page stuff %%%%%%%%%%%%%%%%%%%%%%%%%%

%%%%%%%%%%%%%%%%%%%%%%%%%%%%%%%%%%%%%%%%%%%%%%%%%%%%%%%%%%%%%%%%%%%%%%%%%%%%%%%%%%%%%%%%%%%%%%%%%%%%%%%%%%%%%%%%%%% journal ideas: ???

\title{A unified construction of semiring-homomorphic graph invariants}

\author{Tobias Fritz}

\address{Perimeter Institute for Theoretical Physics, Waterloo, Ontario, Canada}
\email{tfritz@pitp.ca}



\thanks{\textit{Acknowledgements.} Most of this work was conducted while the author was affiliated with the Max Planck Institute for Mathematics in the Sciences. We thank Chris Cox, Chris Godsil, Matilde Marcolli, David Roberson, and Jeroen Zuiddam for various combinations of useful discussions and detailed feedback on a draft.}

\begin{abstract}
It has recently been observed by Zuiddam that finite graphs form a preordered commutative semiring under the graph homomorphism preorder together with join and disjunctive product as addition and multiplication, respectively. This led to a new characterization of the Shannon capacity $\Theta$ via Strassen's Positivstellensatz: $\Theta(\bar{G}) = \inf_f f(G)$, where $f : \Graph \to \R_+$ ranges over all monotone semiring homomorphisms.

Constructing and classifying graph invariants $\Graph \to \R_+$ which are monotone under graph homomorphisms, additive under join, and multiplicative under disjunctive product is therefore of major interest. We call such invariants \emph{semiring-homomorphic}. The only known such invariants are all of a \emph{fractional} nature: the fractional chromatic number, the projective rank, the fractional Haemers bounds, as well as the Lov\'asz number (with the latter two evaluated on the complementary graph). Here, we provide a unified construction of these invariants based on \emph{linear-like semiring families} of graphs. Along the way, we also investigate the additional algebraic structure on the semiring of graphs corresponding to fractionalization.

Linear-like semiring families of graphs are a new notion of combinatorial geometry different from matroids which may be of independent interest.
\end{abstract}

\maketitle

\tableofcontents

\section{Introduction}

The Shannon capacity of graphs is a notorious graph invariant with high relevance to information theory~\cite{KO}. For technical convenience we state the definition of the Shannon capacity of the complementary graph,
\[
	\Theta(\bar{G}) := \sup_{n\in\N} \sqrt[n]{\omega(G^{\ast n})}.
\]
Here, $\omega$ is the clique number, and $G^{\ast n}$ is the $n$-fold disjunctive product of $G$ with itself. Although the very definition of $\Theta(\bar{G})$ provides an algorithm for computing a sequence of lower bounds converging to $\Theta(\bar{G})$, already finding an algorithm computing a convergent sequence of upper bounds seems to be an open problem. In other words, the computability of $\Theta(\bar{G})$ itself is open. The simplest graph for which the Shannon capacity is not known is the 7-cycle~\cite{7cycle}.

As an application of Strassen's Positivstellensatz~\cite[Corollary~2.6]{strassen}, it was recently shown by Zuiddam how to obtain a tight family of upper bounds on the Shannon capacity, improving substantially on our earlier characterization~\cite[Example~8.25]{fritz}.

\begin{thm}[Zuiddam]
The Shannon capacity satisfies
\beq
\label{shannon_capacity1}
	\Theta(\bar{G}) = \inf_{\eta} \eta(G),
\eeq
where $\eta$ ranges over all graph invariants $\eta : \Graph \to \R_+$ which satisfy the normalization $\eta(K_1) = 1$ as well as the following conditions:
\begin{itemize}
	\item $\eta$ is monotone under graph homomorphisms:
		\[
			G\to H \qquad\Longrightarrow\qquad \eta(G) \leq \eta(H);
		\]
	\item $\eta$ is additive under graph joins:
		\[
			\eta(G + H) = \eta(G) + \eta(H);
		\]
	\item $\eta$ is multiplicative under disjunctive products:
		\[
			\eta(G \ast H) = \eta(G) \eta(H).
		\]
\end{itemize}
\label{Theta_inf1}
\end{thm}

Here, $\Graph$ is the set of all (isomorphism classes of) finite graphs. So if a complete classification of such graph invariants was available, our understanding of the Shannon capacity would improve dramatically. For example, if these invariants could even be enumerated algorithmically, then we would have an algorithm for computing $\Theta$.

Graph invariants satisfying these three conditions are therefore of major interest, and we dedicate the present paper to their study. It is convenient to introduce a new term for them; we call them \textbf{semiring-homomorphic} graph invariants, since the preservation of addition and multiplication is what characterizes semiring homomorphisms. The most basic examples of semiring-homomorphic graph invariants are the fractional chromatic number $\chi_f$ and the complementary Lov\'asz number $G\mapsto \vartheta(\bar{G})$~\cite{lovasz}. Zuiddam observed that the only known additional examples are the projective rank~\cite{MR} and the fractional Haemers bounds of~\cite{bukh_cox}, both also applied to the complement.

What we achieve in this paper is a general construction of semiring-homomorphic graph invariants which recovers all the known examples. Curiously, all of these are of a \emph{fractional} nature, and in particular are monotone not only under graph homomorphisms, but even under the more abundant fractional graph homomorphisms (in our sense, \cref{frac_mon}). Correspondingly, our considerations involve fractional graph theory~\cite{SU} in an essential way.

We now give a brief outline of our construction. Many graph invariants $\eta$ are constructed by \emph{labelling} the vertices of a graph by certain other objects, in such a way that the labels associated to adjacent vertices satisfy a certain relation. Essentially by definition, this type of labelling is secretly a graph homomorphism, namely to the graph with allowed labels as vertices and the specified relation as adjacency. It has been well-recognized that fractionalizing a graph invariant $\eta$ amounts to considering vertex labels given by \emph{sets} of labels of cardinality $d$, dividing the resulting invariant by $d$, and taking the limit $d\to\infty$; see e.g.~\cite[Section~1.2]{HPRS}. Here we formalize the general pattern behind this construction: if $F = (F_n)_{n\in\N}$ is a sequence of graphs, then the \emph{$F$-number} $\graphinv{F}$ is the graph invariant given by
\[
	\graphinv{F}(G) := \min\ \{\: n\in\N \mid G \to F_n \:\},
\]
where $G \to F_n$ means that there exists a graph homomorphism from $G$ to $F_n$, i.e.~a labelling of the vertices of $G$ by the vertices of $F_n$ such that adjacent vertices have labels that are adjacent in $F_n$. The corresponding fractional invariant is
\[
	\graphinvfrac{F}(G) := \inf\ \left\{\: \frac{n}{d} \:\Big|\: G \to F_n/d \:\right\},
\]
where $F_n/d$ for $d,n\in\Nplus$ is the graph of $d$-cliques in $F_n$ (\cref{def_frac}). This formalizes the intuitive idea that a $d$-fractional labelling assigns to each vertex a set of labels of cardinality $d$. It specializes to the definition of the fractional chromatic number in the case where $F_n = K_n$ is the complete graph (\cref{frac_chrom}).

We find that this definition makes the resulting fractional invariant $\graphinvfrac{F}$ into a semiring-homomorphic invariant as soon as the family of graphs $(F_n)$ satisfies a number of conditions that make it into a \emph{linear-like semiring family} (\cref{srfam,ll}). The role played by these graph families is akin to that of the family of Euclidean spaces $(\R^n\setminus\{0\})$, with orthogonality as adjacency, and has a strong geometrical flavour. We therefore think of linear-like semiring families of graphs as a new concept of combinatorial geometry, distinct from matroid theory, which could be of independent interest.

\subsection*{Summary}

In \cref{algebra}, we introduce our setting, emphasizing that our graphs can also be infinite. We recall the best known version of Strassen's Positivstellensatz as \cref{spss_better} and apply it to $\Graph$, the preordered semiring of graphs, in order to state Zuiddam's \cref{Theta_inf1} as \cref{Theta_inf2} and the more general \cref{rate_formula} on rates. In \cref{sec_frac}, we consider operations of blowup and fractionalization of graphs as additional algebraic structure on $\Graph$. In \cref{sec_srfams}, we introduce semiring families of graphs $F = (F_n)$ and prove \cref{mainthm} as our main result, stating in particular that $\graphinvfrac{F}$ is semiring-homomorphic. The proof mirrors the more specific arguments of Bukh and Cox~\cite{bukh_cox} in our more general situation. In \cref{sec_recover}, we show how this recovers all known examples of semiring-homomorphic graph invariants as particular instances, thereby also giving the first (and so far only) examples of linear-like semiring families. In \cref{sec_ops}, we end with a discussion of some open problems.

\section{The preordered commutative semiring of graphs}
\label{algebra}

\begin{defn}
\label{def_graph}
Throughout this paper, a \emph{graph} $G$ is an undirected simple graph $G = (V(G), E(G))$, where both the set of vertices $V(G)$ and the set of edges $E(G)$ may be infinite. As usual, we denote adjacency $(v,w) \in E(G)$ for vertices $v,w\in V(G)$ by $v\sim w$.
\end{defn}

So in contrast to~\cite{zuiddam}, we also include \emph{infinite} graphs in some of our considerations. We do this in order to obtain e.g.~the Lov\'asz number (\cref{lovasz}) from our construction, although it is open whether infinite graphs are really necessary in order to do so (\cref{fin_vs_infin}).

We now get to the definition of $\Graph$ (resp.~$\FGraph$), the preordered commutative semiring of (finite) graphs. The essential pieces of structure on $\Graph$ and $\FGraph$ are the following standard definitions~\cite{SU,GR}:

\begin{defn}
Let $G$ and $H$ be graphs.
\begin{enumerate}
\item A \emph{graph homomorphism} $f : G\to H$ is a map $f : V(G) \to V(H)$ such that if $v\sim w$ in $G$, then also $f(v) \sim f(w)$ in $H$.
\item The \emph{graph join} $G + H$ is the graph with vertex set the disjoint union $V(G + H) := V(G) \sqcup V(H)$, the induced edges on each summand, and $v\sim w$ for every $v\in V(G)$ and $w\in V(H)$.
\item The \emph{disjunctive product} $G \ast H$ is the graph with vertex set the cartesian product $V(G\ast H) := V(G) \times V(H)$, and $(v,w) \sim (v',w')$ if and only if $v\sim v'$ or $w\sim w'$.
\item The \emph{lexicographic product} $G\ltimes H$ is the graph with vertex set the cartesian product $V(G\ltimes H) := V(G)\times V(H)$, and $(v,w) \sim (v',w')$ if and only if $v\sim v'$, or $v = v'$ and $w \sim w'$.
\end{enumerate}
\end{defn}

We write $G\to H$ to denote the statement that there is a graph homomorphism from $G$ to $H$. For example, we always have $G\ltimes H \to G \ast H$. We say that two graphs $G$ and $H$ are \emph{homomorphically equivalent} if $G\to H$ and $H\to G$, and denote this by $G\simeq H$; we occasionally also write $G\cong H$ to denote graph isomorphism. We will frequently use the \emph{clique number}
\[
	\omega(G) := \sup\, \{\: n\in\N \mid K_n \to G \:\}
\]
and the \emph{chromatic number}
\[
	\chi(G) := \inf\, \{\: n\in\N \mid G \to K_n \:\}.
\]
Both of these invariants are additive in the sense that $\omega(G + H) = \omega(G) + \omega(H)$ and $\chi(G + H) = \chi(G) + \chi(H)$. For infinite graphs, either or both of these may be infinite.

Besides working with potentially infinite graphs, we additionally deviate in our conventions slightly from~\cite{zuiddam} in that our definitions correspond to those of Zuiddam under taking graph complements. Our reasons are that this allows us not to take complements when talking about graph homomorphisms, and that this results in a more straightforward information-theoretic interpretation in terms of zero-error communication~\cite{shaye_wang}.

We write $\Graph$ for the collection of all isomorphism classes of graphs. $\Graph$ becomes a preordered set\footnote{Recall that a preorder ``$\leq$'' is a reflexive and transitive relation, not necessarily antisymmetric.} if we put $G\leq H$ if and only if there is a graph homomorphism $G\to H$. Even upon considering finite graphs only, the resulting preordered set is known to be highly complex~\cite[Section~IV.3]{PT}. In order to keep our notation intuitive, we will continue to write $G\to H$ in order to indicate the existence of a graph homomorphism from $G$ to $H$ and avoid the notation $G\leq H$. We refer to~\cite{zuiddam,our_SPSS} for the definition of preordered commutative semiring. The following is again due to Zuiddam~\cite{zuiddam} in the case of finite graphs.

\begin{prop}
With graph join as addition and disjunctive product as multiplication, $\Graph$ is a preordered commutative semiring.
\end{prop}

\begin{proof}
Straightforward to check; e.g.~the distributivity of multiplication over addition is given by the canonical isomorphism
\[
	G_1 \ast H + G_2 \ast H \stackrel{\cong}{\longrightarrow} (G_1 + G_2)\ast H,
\]
which is given by the identity on the level of vertices for all graphs $G_1$, $G_2$ and $H$.
\end{proof}

It is noteworthy that our version of $\Graph$ is technically a proper class rather than a set, since it contains graphs of arbitrarily large cardinality\footnote{This is because if we only consider graphs of cardinality at most $\kappa$ for some infinite cardinal $\kappa$, then there are at most $2^\kappa\times 2^{\kappa \times \kappa} = 2^\kappa$ isomorphism classes of graphs. Here, the first factor $2^\kappa$ arises from choosing a set of vertices, and the second factor $2^{\kappa\times \kappa}$ from choosing a set of edges.}. This type of situation is closely familiar from category theory, and here we also adopt the standard workaround using Grothendieck universes~\cite[Section~3.3]{Schu}, which lets us ignore this issue from now on.

We also write $\FGraph \subseteq \Graph$ for the preordered subsemiring of finite graphs, with respect to the restricted preorder relation and algebraic operations.

The objects of study of this paper are the following:

\begin{defn}
A \emph{semiring-homomorphic graph invariant} is an order-preserving semiring homomorphism $\eta : \FGraph \to \R_+$.
\end{defn}

Concretely, a \emph{semiring-homomorphic graph invariant} $\eta$ therefore assigns a number $\eta(G)$ to every graph $G\in\FGraph$ such that $\eta(K_1) = 1$, and such that the following hold for any two finite graphs $G,H\in\FGraph$:
\begin{itemize}
\item If $G\to H$, then $\eta(G) \leq \eta(H)$;
\item $\eta(G + H) = \eta(G) + \eta(H)$;
\item $\eta(G\ast H) = \eta(G)\, \eta(H)$.
\end{itemize}
The set of semiring-homomorphic graph invariants corresponds to the space of points defined over $\R_+$ of the \emph{real spectrum}\footnote{Strassen~\cite{strassen} has coined the term \emph{asymptotic spectrum} for this kind of set in light of condition~\eqref{large_power}. Due to the equivalence with the other conditions in \cref{spss_better}, which are arguably less of an asymptotic nature, and also due to the close relation with the real spectrum as used in real algebraic geometry~\cite[Section~2.4]{marshall}, we no longer find Strassen's terminology to be optimally descriptive.} of $\FGraph$. We study the known examples of semiring-homomorphic graph invariants as \cref{frac_chrom,symm_ex,symp_ex,lovasz}.

The purpose of this paper is to contribute to the problem of \emph{classifying} all semiring-homomorphic graph invariants, a project which has been initiated by Zuiddam~\cite{zuiddam}. It is a curious fact that all known semiring-homomorphic graph invariants are of a fractional nature; the fact that fractional graph theory plays an important role in this paper is merely due to this heuristic observation, and we do not know whether fractional graph theory is intrinsically important for the desired classification. In particular, it is open whether all semiring-homomorphic graph invariants are also monotone under fractional graph homomorphisms (\cref{graph_vs_fgraph}).

The rest of this section is mainly motivational, and dedicated to illustrating the relevance of trying to achieve a classification of semiring-homomorphic graph invariants. We begin by recalling a variant of Strassen's Positivstellensatz recently obtained by us in~\cite[Corollary~2.14]{our_SPSS}. This theorem adds two further equivalent characterization to Zuiddam's improved version~\cite[Theorem~2.2]{zuiddam} over Strassen's original result.

\begin{thm}[\cite{our_SPSS}]
\label{spss_better}
Let $S$ be a preordered commutative semiring such that $\N\subseteq S$ via the unique semiring homomorphism $\N\to S$, and such that for every nonzero $x\in S$ there is $n\in\N$ with $x \leq n$ and $1 \leq nx$. Then for nonzero $x,y\in S$, the following are equivalent:
\begin{enumerate}
	\item\label{a} $f(x) \geq f(y)$ for every monotone semiring homomorphism $f : S \to \R_+$.
\item\label{closure_order_new} For every $\eps > 0$, there are $m,n\in\Nplus$ and nonzero $z\in S$ such that $m \leq \eps n$ and
\[
	n\, z\, x + m\, z \geq n\, z\, y.
\]
\item\label{cat_order_new} For every $\eps > 0$, there are $m,n\in\Nplus$ and nonzero $z\in S$ such that $\frac{m}{n} \leq 1 + \eps$ and
\[
	m\, z\, x \geq n\, z\, y.	
\]
\item\label{d} For every $\eps > 0$ there are $k,n\in\Nplus$ such that $k\leq \eps n$ and
\beq
\label{large_power}
	2^k \, x^n \geq y^n.
\eeq
\end{enumerate}
Moreover, if $f(x) > f(y)$ for every monotone semiring homomorphism $f : S \to \R_+$, then there are $n\in\N$ and nonzero $z,w\in S$ with $zx + w \geq zy + w$.
\end{thm}

In the case of $\FGraph$, the canonical homomorphism $\N\to\FGraph$ takes $n\in\N$ to the complete graph $K_n$, so that the ``scalars'' in $\FGraph$ are precisely the complete graphs. The product $n\, G$ is equal to the $n$-fold join of $G$ with itself. The boundedness conditions $x\leq n$ and $1 \leq nx$ therefore mean that for every $G\in\FGraph$ with at least one vertex, there is $n\in\N$ such that $G\to K_n$ and $K_1 \to n\, G$. While the latter condition simply states that $n\, G$ has at least one vertex, the former corresponds to having finite chromatic number, which is trivial for a finite graph. Therefore $\FGraph$ is a preordered semiring satisfying the hypotheses of \cref{spss_better}.

We now focus on the equivalence of conditions~\ref{a} and~\ref{d} in the case of $\FGraph$; the implications of the other statements of \cref{spss_better} for $\FGraph$ have not yet been explored.

%The following was shown by Zuiddam~\cite{zuiddam} for the case of finite graphs, as a consequence of Strassen's Positivstellensatz for finite graphs. Our generalization to infinite graphs is at most a minor modification, so that we still attribute the result to Zuiddam.

\begin{thm}[Zuiddam]
\label{Theta_inf2}
The Shannon capacity satisfies
\beq
\label{shannon_capacity2}
	\Theta(\bar{G}) = \inf_{\eta} \eta(G),
\eeq
where $\eta$ ranges over all semiring-homomorphic graph invariants.
\end{thm}

Before the proof, we note that it is also of interest to compute \emph{rates} for finite graphs $G,H\in\FGraph$ with at least one vertex,
\[
	R(G\to H) := \sup\ \left\{\ \frac{m}{n} \:\bigg|\: H^{\ast m} \to G^{\ast n} \ \right\},
\]
as originally defined in~\cite[Examples~8.1 and 8.18]{fritz}, and subsequently rediscovered as the \emph{graph information ratio} in~\cite{shaye_wang}, where an operational interpretation in terms of zero-error information theory has been given. As a special case, we have $R(G \to K_2) = \log_2 \Theta(\bar{G})$~\cite[Example~8.1]{fritz}. These rates can also be computed in terms of an optimization over semiring-homomorphic graph invariants:

\begin{thm}
\label{rate_formula}
If $G$ has at least one edge and $H$ at least one vertex, then
\[
	R(G \to H) = \inf_\eta \, \frac{\log \eta(G)}{\log \eta(H)},
\]
where $\eta$ ranges over all semiring-homomorphic graph invariants.
\end{thm}

This is in the spirit of the \emph{rate formulas} for resource efficiency in the sense of~\cite[Theorem~8.24]{fritz}. \cref{Theta_inf2} follows from this\footnote{The degenerate cases where $G$ has no edge or even no vertex must be treated separately, but it is obvious that \cref{Theta_inf2} holds for them.} upon using $R(G\to K_2) = \log_2 \Theta(\bar{G})$ as noted above together with $\eta(K_2) = 2$ for all $\eta$. 

\begin{proof}
This holds for regularized rates as per~\cite[Corollary~2.17]{our_SPSS}, which coincide with the $R(G\to H)$ as per~\cite[Example~8.18]{fritz}.
\end{proof}

By \cref{Theta_inf2}, if a classification of semiring-homomorphic graph invariants was available in such a way that one could enumerate them algorithmically, this would resolve the long-standing open question of computability of $\Theta$ in the positive. So far, only few semiring-homomorphic graph invariants are known at all; we will discuss them in \cref{sec_recover}, explaining how they are instances of the general construction given in \cref{sec_srfams}, for which we prepare in the next section.

\begin{rem}
	Both \cref{Theta_inf2,rate_formula} are still correct, with the same proofs remaining valid, if we include in the definition of $\FGraph$ not only the finite graphs, but all graphs with finite chromatic number.
\label{finite_chi}
\end{rem}

\section{The algebraic structure of blowup and fractionalization}
\label{sec_frac}

The following considerations make explicit the algebraic structure of fractionalization, in the general case of possibly infinite graphs. They apply to the finite case in particular in the sense that all the following constructions preserve finiteness. In this context, some of the following observations also appear in the proofs in Zuiddam's thesis~\cite[Section~3.3.2]{zuiddam_thesis}, in particular~\eqref{blow_distribute} and~\cref{unidistributive}.

\begin{defn}
\label{def_blowup}
For a graph $G$ and $d\in\Nplus$, the \emph{$d$-fold blowup} is given by the lexicographic product $G\ltimes K_d$, and we denote it by $G\ltimes d$.
\end{defn}
	
In other words, $G\ltimes d$ arises from $G$ by replacing each vertex $v\in V(G)$ by $d$ adjacent copies $v_1,\ldots,v_d\in V(G\ltimes d)$. It is easy to see that the map $G\mapsto G\ltimes d$ is functorial in the sense that it takes graph homomorphisms to graph homomorphisms, and therefore also preserves homomorphic equivalence.

It is important to distinguish $G\ltimes d$ from the $d$-fold scalar multiple of $G$ in $\Graph$, which is $d\,G = K_d\ast G$. The blowup approximates $d\, G$ in that $G\ltimes d \to K_d \ast G$, while the converse does generally not hold:

\begin{ex}
\label{blowup_not_scalar}
Let $C_5$ be the $5$-cycle. Then $2\,C_5 \not\to C_5\ltimes 2$ since $\chi(2\,C_5) = \chi(K_2\ast C_5) = 6$, while $\chi(C_5\ltimes 2) = 5$. (This is essentially the well-known observation that the fractional chromatic number of $C_5$ is strictly smaller than its chromatic number, since $\chi$ always commutes with scalar multiplication, $\chi(d\,G) = d\,\chi(G)$. This is easy to see directly or from \cref{eta_additive} and \cref{frac_chrom}.)
\end{ex}

\begin{defn}
Given a graph $G$, its \emph{power graph} $2^G$ has as vertices the subsets of $V(G)$, with adjacency $S\sim T$ for two subsets $S,T\subseteq V(G)$ if and only if $S\cap T = \emptyset$ and $s\sim t$ for every $s\in S$ and $t\in T$.
\end{defn}

\begin{defn}
\label{def_frac}
For a graph $G$ and $d\in\N$, the \emph{$d$-fractionalization} $G/d$ is the induced subgraph of $2^G$ on the set of $d$-cliques in $G$.
\end{defn}

For example, the $d$-fractionalization of a complete graph is the corresponding Kneser graph, $K_n / d = KG_{n,d}$. Again the operation $G\mapsto G/d$ is obviously functorial in $G$ under graph homomorphisms, and therefore preserves homomorphic equivalence.

The notation $G/d$ is motivated by the following observation:

\begin{lem}
\label{adj}
The $d$-fold blowup and the $d$-fractionalization are Galois adjoints:
\[
	G\ltimes d \to H \quad \Longleftrightarrow \quad G \to H/d.
\]
\end{lem}

\begin{proof}
	It is straightforward to see that either type of homomorphism can be reinterpreted as the other.
\end{proof}

The adjunction implies that $G \to (G\ltimes d) / d$ and $(H/d) \ltimes d \to H$, although the converse relations generally do not hold:

\begin{ex}
Starting with $K_3$, we have $(K_3\ltimes 2) / 2 \cong K_6 / 2 \cong KG_{6,2}$, and $\chi(KG_{6,2}) = 4$ by the Lov\'asz--Kneser theorem~\cite[Chapter~38]{thebook}. Therefore $(K_3\ltimes 2)/2 \not\to K_3$. For the other statement, since $C_5 / 2$ is an edgeless graph, it is trivial that $C_5 \not\to (C_5 / 2)\ltimes 2$.
\label{monads}
\end{ex}

\begin{lem}
\label{Naction}
For all $G$ and $d,d'\in\Nplus$,
\[
	(G\ltimes d)\ltimes d' \cong G\ltimes dd', \qquad G/d'/d \simeq G/(dd').
\]
\end{lem}

\begin{proof}
The first isomorphism is easy to check directly. The second homomorphic equivalence then follows from \cref{adj} by the uniqueness of adjoints up to equivalence, together with composition of adjoints: $(-\ltimes d)\ltimes d'\dashv (-)/d'/d$ and $(-\ltimes d)\ltimes d' \cong (-) \ltimes dd' \dashv (-)/(dd')$.
\end{proof}

Despite the homomorphic equivalence, the isomorphism $G/d'/d \cong G/(dd')$ does generally not hold, since there are generally many ways of partitioning a $(dd')$-clique into a $d$-clique of $d'$-cliques.

\begin{lem}
\label{plus_distribute}
For any graphs $G$ and $H$ and $d\in\Nplus$,
\begin{align}
	G \ltimes d + H \ltimes d & \cong (G + H) \ltimes d, \label{blow_distribute} \\
	G / d + H / d & \to (G + H) / d, \label{frac_distribute}
\end{align}
but in general $(G + H) / d \not\to G / d + H / d$.
\end{lem}

\begin{proof}
Again the first isomorphism is straightforward to check\footnote{It is also an instance of the general categorical fact that left adjoints preserve colimits.}. The second homomorphism $G / d + H / d \to (G + H) / d$ is clear since the images of the two canonical homomorphisms $G / d \to (G + H) / d$ and $H / d \to (G + H) / d$ are pairwise adjacent.

In the other direction, taking $G = H = K_1$ and $d = 2$ gives a (rather degenerate) counterexample, since $K_1 / 2 \cong \emptyset$, while $K_2 / 2 \cong K_1$.
\end{proof}

Let us emphasize again that the operation $(-)\ltimes d$ does not coincide with scalar multiplication in $\Graph$. More specifically, although
\[
	G\ltimes (d + d') \to G\ltimes d + G\ltimes d'
\]
holds, the converse does generally not, as e.g.~the example $G = C_5$ and $d = d' = 1$ shows (\cref{blowup_not_scalar}).

Finally, there is a compatibility inequality between the blow-up and the disjunctive product:

\begin{lem}
\label{unidistributive}
For any graphs $G$ and $H$ and $d\in \N$,
\[
	(G\ast H)\ltimes d \to (G\ltimes d) \ast H.
\]
but the converse does not hold in general.
\end{lem}

\begin{proof}
The adjacency is given by $(i,g,h) \sim (i',g',h')$ on the left-hand side if and only if
\[
	h \sim h' \lor g \sim g' \lor (g = g' \land h = h' \land i \neq i')
\]
and on the right-hand side if and only if
\[
	h \sim h' \lor g \sim g' \lor (g = g' \land i \neq i').
\]
It is clear that the first implies the second.

As a counterexample to the other direction, we have $(K_2\ltimes 2)\ast C_5 \not\to (K_2\ast C_5)\ltimes 2$. One way to see this is to note that $\chi( (K_2\ltimes 2) \ast C_5) = \chi(4\, C_5) = 4 \cdot 3$, while $\chi((K_2\ast C_5)\ltimes 2\ltimes 2) = 10$, where the latter is per explicit computation in \textsc{Sage}.
\end{proof}

We finish our investigations of the algebraic structure of fractionalization with a result on the compatibility of blow-up with fractionalization:

\begin{lem}
For a graph $G$ and $d,d'\in\Nplus$, we have $(G/d') \ltimes d \to (G\ltimes d) / d'$.
\label{blowup_vs_frac}
\end{lem}

\begin{proof}
We already know $(G/d')\ltimes d' \to G$, which implies $(G/d') \ltimes dd' \to G\ltimes d$ by functoriality of blow-up and \cref{Naction}. Another application of \cref{Naction} and the adjunction of \cref{adj} gives indeed $(G/d')\ltimes d \to (G\ltimes d) / d'$.
\end{proof}

\section{Semiring families of graphs and the linear-like condition}
\label{sec_srfams}

The essential ingredients of our upcoming construction of semiring-homomorphic graph invariants will be semiring families of graphs:

\begin{defn}
\label{srfam}
A \emph{semiring family} of graphs is given by a sequence of graphs $(F_n)_{n\in\N}$ such that $F_0 = \emptyset$ and $F_1\neq\emptyset$ and for all $n,m\in\N$,
\beq
\label{srfam_homs}
	F_n + F_m \to F_{n + m}, \qquad F_n * F_m \to F_{nm}.
\eeq
\end{defn}

These two operations are analogues of the addition and multiplication on a single semiring; we could say that they are equivalent to specifying a \emph{lax semiring homomorphism} $\N \to \Graph$ by analogy with lax monoidal functors~\cite[Chapter~3]{AM}. The condition $F_1\neq\emptyset$ then corresponds to unitality in the form $K_1 \to F_1$. We intentionally do not consider the particular homomorphisms implementing~\eqref{srfam_homs} as part of the structure of a semiring family. In particular, we do not require the existence of homomorphisms which would satisfy compatibility equations like associativity or distributivity, since not all of these actually hold e.g.~in the upcoming \cref{lovasz}. Finally, one can also imagine indexing a semiring family by semirings other than $\N$, but we will not consider such more general families here.

Since $F_1\neq\emptyset$, we have $\omega(F_1) \geq 1$. We can then use induction and $F_n + F_1 \to F_{n+1}$ to show that $\omega(F_n) \geq n$ for all $n\in\N$. Next, we show that it's possible to take ``common denominators'':

\begin{lem}
\label{com_denom}
Let $(F_n)$ be a semiring family. Then for every $n,d,m\in\Nplus$,
\[
	F_n/d \to F_{mn}/md.
\]
\end{lem}

\begin{proof}
By \cref{Naction} and functoriality of $-/d$, it is enough to consider the case $d = 1$. By the adjunction of \cref{adj}, we only need to prove $F_n\ltimes m \to F_{mn}$. This arises as the composite
\[
	F_n\ltimes m \to K_m \ast F_n \to F_m \ast F_n \to F_{mn}. \qedhere
\]
\end{proof}

We will prove in \cref{mainthm} that every semiring family which satisfies an additional \emph{linear-like} condition gives rise to a semiring-homomorphic graph invariant by fractionalization. In general, there are three basic invariants that can be constructed from a semiring family:

\begin{defn}
Let $F = (F_n)$ be a semiring family and $G\in\FGraph$.
\begin{enumerate}
\item The \emph{$F$-number} of $G$ is
\beq
\label{eta}
	\graphinv{F}(G) := \min \{\: n\in\N \mid G \to F_n \: \}.
\eeq
\item The \emph{fractional $F$-number} of $G$ is
\beq
\label{etaf}
	\graphinvfrac{F}(G) := \inf \left\{\: \frac{n}{d} \Bigm| G \to F_n/d \:\right\}.
\eeq
\item The \emph{asymptotic $F$-number} of $G$ is
\beq
\label{etainfty}
\graphinvasymp{F}(G) := \inf_n \sqrt[n]{\graphinv{F}(G^{\ast n})}
\eeq
\end{enumerate}
\end{defn}

Here, the relation of $\graphinv{F}$ and $\graphinvfrac{F}$ can also be considered an instance of general fractionalization of graph parameters~\cite[(3.2)]{zuiddam_thesis}.

If we think of a semiring family $(F_n)$ as a combinatorial generalization of finite-dimensional inner product spaces over a field, with orthogonality as adjacency, then a homomorphism $G\to F_n$ is a combinatorial generalization of an orthogonal representation (or vector colouring), so that $\graphinv{F}(G)$ has the flavour of a vector chromatic number. In the case $F_n = K_n$ (\cref{frac_chrom}), we recover usual colourings and the usual chromatic number $\chi(G)$.

All three invariants are finite since $G$ is finite, and therefore has finite chromatic number, which implies that there is $n\in\N$ with $G\to K_n\to F_n$. Also all three invariants are monotone under graph homomorphisms by construction. Moreover:

\begin{prop}
\label{subinvs}
All three invariants $\graphinv{F}^\ast \in\{\graphinv{F}, \graphinvfrac{F}, \graphinvasymp{F}\}$ are subadditive under joins,
\[
	\graphinv{F}^\ast(G + H) \leq \graphinv{F}^\ast(G) + \graphinv{F}^\ast(H)
\]
and submultiplicative under lexicographic and disjunctive products,
\[
	\graphinv{F}^\ast(G\ltimes H) \leq \graphinv{F}^\ast(G \ast H) \leq \graphinv{F}^\ast(G) \, \graphinv{F}^\ast(H).
\]
\end{prop}

\begin{proof}
The inequality $\graphinv{F}^\ast(G\ltimes H) \leq \graphinv{F}^\ast(G\ast H)$ is by monotonicity of $\graphinv{F}^\ast$ and $G\ltimes H \to G\ast H$, so that it only remains to prove subadditivity and submultiplicativity with respect to the disjunctive product.

This is clear for the $F$-number $\graphinv{F}$, as it follows directly from the definition of a semiring family.

We prove the subadditivity for the fractional $F$-number $\graphinvfrac{F}$, which follows upon choosing common denominators as follows. For $\eps > 0$, choose $n,d,n',d'\in\N$ with $d,d' > 0$ such that we have $\eps$-approximations to the defining infima of $\graphinvfrac{F}$, namely
\[
	\graphinvfrac{F}(G) \geq \frac{n}{d} - \eps, \qquad \graphinvfrac{F}(H) \geq \frac{n'}{d'} - \eps,
\]
arising from $G \to F_n/d$ and $H \to F_{n'}/d'$. By \cref{com_denom}, we can assume $d = d'$ without loss of generality. But then we have $G \ltimes d\to F_n$ and $H\ltimes d \to F_{n'}$ by the Galois adjunction of \cref{adj}. Then $G\ltimes d + H\ltimes d \to F_n + F_{n'} \to F_{n + n'}$, and hence $(G + H)\ltimes d \to F_{n + n'}$ by the distributivity of \cref{plus_distribute}. Now $G + H \to F_{n+n'}/d$ again by the adjunction, which gives $\graphinvfrac{F}(G + H) \le \frac{n + n'}{d} \leq \graphinvfrac{F}(G) + \graphinvfrac{F}(H) + 2\eps$. The claim follows in the limit $\eps \to 0$.

The submultiplicativity of $\graphinvfrac{F}$ works similarly to the subadditivity, using the same data. We have $G \ltimes d\to F_n$ and $H\ltimes d' \to F_{n'}$, and therefore
\[
	(G \ast H)\ltimes dd' \to (G\ltimes d) \ast (H\ltimes d') \to F_n \ast F_{n'} \to F_{nn'},
\]
where the first homomorphism is by \cref{Naction,unidistributive}. Therefore
\[
	\graphinvfrac{F}(G \ast H) \leq \frac{nn'}{dd'} \leq \left( \graphinvfrac{F}(G) + \eps \right) \left( \graphinvfrac{F}(H) + \eps \right),
\]
so that submultiplicativity follows in the limit $\eps \to 0$.

For $\graphinvasymp{F}$, we first note that the infimum $\inf_n \sqrt[n]{\graphinv{F}(G^{\ast n})}$ is a limit by Fekete's lemma and submultiplicativity of $\graphinv{F}$. Using this, submultiplicativity of $\graphinvasymp{F}$ is then clear again by submultiplicativity of $\graphinv{F}$,
\[
	\sqrt[n]{\graphinv{F}( (G\ast H)^{\ast n} )} = \sqrt[n]{\graphinv{F}(G^{\ast n} \ast H^{\ast n})} \leq \sqrt[n]{\graphinv{F}(G^{\ast n})} \sqrt[n]{\graphinv{F}(H^{\ast n})},
\]
since then the inequality also holds in the limit. Finally we treat subadditivity of $\graphinvasymp{F}$. For given $\eps > 0$, we fix $k\in\N$ such that for all $m\geq k$,
\[
	\graphinvasymp{F}(G) \geq \sqrt[m]{\graphinv{F}(G^{\ast m})} - \eps, \qquad 
	\graphinvasymp{F}(H) \geq \sqrt[m]{\graphinv{F}(H^{\ast m})} - \eps.
\]
Now we have that $G^{\ast j} \to G^{\ast k}$ for $j \leq k$, and therefore for sufficiently large $n$,
\begin{align*}
	\sqrt[n]{\graphinv{F}( (G + H)^{\ast n} )} & = \sqrt[n]{\graphinv{F}\left(\sum_{j=0}^n \binom{n}{j} G^{\ast j} \ast H^{\ast (n-j)} \right)} \\
	& \leq \sqrt[n]{\sum_{j=0}^n \binom{n}{j} \graphinv{F}(G^{\ast j}) \graphinv{F}(H^{\ast (n-j)}) } \\
	& \leq \sqrt[n]{k \binom{n}{k} (\graphinv{F}(G^{\ast k}) \graphinv{F}(H^{\ast n}) + \graphinv{F}(G^{\ast n}) \graphinv{F}(H^{\ast k})) + \sum_{j=k}^{n-k} \binom{n}{j} \graphinv{F}(G^{\ast j}) \graphinv{F}(H^{\ast (n-j)}) } \\
	& \leq \sqrt[n]{O(\mathrm{poly}(n)) \, ((\graphinvasymp{F}(H) + \eps)^n + (\graphinvasymp{F}(G) + \eps)^n ) + \left( \graphinvasymp{F}(G) + \graphinvasymp{F}(H) + 2\eps \right)^n } \\
\end{align*}
Since $\lim_{n\to\infty} \sqrt[n]{\alpha^n + \beta^n + \gamma^n} = \max(\alpha,\beta,\gamma)$, the expression is dominated by the third exponential, and we conclude $\graphinvasymp{F}(G + H) \leq \graphinvasymp{F}(G) + \graphinvasymp{F}(H) + 2\eps$ by taking the limit $n\to\infty$. The claim now follows as $\eps \to 0$.
\end{proof}

Next, we will introduce our \emph{linear-like} condition on a semiring family and prove that this makes $\graphinvfrac{F}$ additive under joins and multiplicative under disjunctive products, and therefore semiring-homomorphic. Introducing this condition requires a bit more preparation.

For $S\subseteq V(G)$, we write $S^\perp$ for the set of all vertices that are adjacent to all vertices in $S$. The map $S\mapsto S^\perp$ implements a contravariant Galois adjunction, meaning that $S \subseteq T^\perp$ if and only if $T\subseteq S^\perp$ if and only if $S$ and $T$ are pairwise adjacent. This is equivalent to adjacency $S\sim T$ in the power graph $2^G$. It follows that the map $S\mapsto S^{\perp\perp}$ is a closure operation. We call those sets $S$ which satisfy $S^{\perp\perp} = S$ \emph{flats}, by analogy with flats in matroid theory. Since $S \subseteq T^\perp$ implies $T^{\perp\perp} \subseteq S^\perp$, we have that $S\sim T$ implies $S \sim T^{\perp\perp}$, and by the same token applied again, also $S^{\perp\perp} \sim T^{\perp\perp}$.

\begin{defn}
The \emph{rank} of a subset $S\subseteq V(G)$ is
\[
	\rk{S} := \omega(S^{\perp\perp}).
\]
\end{defn}

Thus $S$ has rank $\geq r$ if and only if there is an $r$-clique $C$ in $G$ such that for every $T\sim S$ in $2^G$, we also have $T\sim C$. Note that $C$ does not need to be contained in $S$. An interesting special case is when $S$ is itself a clique, where it may happen that $\rk{S} > |S|$, for example if $G$ consists of the clique $S$ and a disjoint clique $C$ larger than $S$ together with a separate cut vertex adjacent to all vertices in both $S$ and $C$, which gives $S^{\perp\perp} = S \cup C$.

\newcommand{\ind}[2]{{#1}|_{#2}}

For a subset $S\subseteq V(G)$, we write $\ind{G}{S}$ for the induced subgraph on $S$.

\begin{defn}
\label{ll}
A sequence of graphs $(F_n)_{n\in\N}$ is \emph{linear-like} if for every flat $S\subseteq V(F_n)$, we have a homomorphic equivalence $\ind{F_n}{S} \simeq F_{\rk{S}}$.
\end{defn}

In some cases, the homomorphic equivalence $\ind{F_n}{S} \simeq F_{\rk{S}}$ can be strengthened to isomorphism $\ind{F_n}{S} \cong F_{\rk{S}}$, but this is not always the case (\cref{symp_ex}).

The definition of linear-like sequence is strongly reminiscent of Brunet's \emph{orthomatroids}~\cite{brunet}, but weaker in a sense. In fact, the closure operation $S\mapsto S^{\perp\perp}$ may not even make $F_n$ into a matroid, even if we have isomorphisms $\ind{F_n}{S} \cong F_{\rk{S}}$. For example in \cref{lovasz}, the two vectors $(1,\pm \sqrt{2},0,\ldots)\in V(F_3)$ have empty orthogonal complement, so that the only flat which contains both is $F_3$ itself, which already has rank $3$. Hence $F_3$ is not a matroid.

However, the linear-like condition implies that if $\phi : G \to F_n$ is any homomorphism, then it factors as
\[
	G \to F_{\rk{\im{\phi}}} \to F_n,
\]
which we will make frequent use of.

\begin{lem}
For a linear-like semiring family $(F_n)$, we have $\omega(F_n) = n$.
\end{lem}

\begin{proof}
$\omega(F_n) \geq n$ follows from $F_1 + \ldots + F_1 \to F_n$, as already noted above. For the other inequality, applying the linear-like condition to the whole graph implies $F_n \simeq F_{\omega(F_n)}$. But since $F_n + F_{\omega(F_n) - n} \to F_{\omega(F_n)}$, we have
\[
	\omega(F_n) = \omega(F_{\omega(F_n)}) \geq \omega(F_n) + \omega(F_{\omega(F_n) - n}) \geq \omega(F_n) + (\omega(F_n) - n).
\]
Therefore $\omega(F_n) \leq n$.
\end{proof}

We now continue assuming that $(F_n)$ is a linear-like semiring family of graphs.

\begin{lem}
\label{addF}
If $F_k + F_\ell \to F_n$, then $k + \ell \le n$.
\end{lem}

\begin{proof}
By monotonicity of $\omega$ together with $\omega(F_k + F_\ell) = k + \ell$ and $\omega(F_n) = n$.
\end{proof}

\begin{lem}
\label{joindecomp}
If $G + H \to F_n$, then there is a decomposition $n = k + \ell$ with $G \to F_k$ and $H \to F_\ell$.
\end{lem}

\begin{proof}
Let $\phi : G + H \to F_n$ be a homomorphism. Then $\im{\phi|_G} \sim \im{\phi|_H}$, and therefore also $\im{\phi|_G}^{\perp\perp} \sim \im{\phi|_H}^{\perp\perp}$. The claim now follows from the linear-like property and \cref{addF}.
\end{proof}

\begin{prop}
	$\graphinv{F}$ and $\graphinvfrac{F}$ are additive under joins.
\label{eta_additive}
\end{prop}

\begin{proof}
We have already shown subadditivity under joins for both invariants in \cref{subinvs}, so that it remains to prove superadditivity. For $\graphinv{F}$, this is by \cref{joindecomp}.

For $\graphinvfrac{F}$, the superadditivity is now similar to the previous arguments: choose $n$ and $d > 0$ such that
\[
	\graphinvfrac{F}(G + H) \ge \frac{n}{d} - \eps,
\]
as witnessed by some homomorphism $G + H \to F_n/d$. Then again $G\ltimes d + H\ltimes d \simeq (G + H)\ltimes d \to F_n$, which decomposes as $G\ltimes d \to F_k$ and $H\ltimes d \to F_\ell$ for suitable $k$ and $\ell$. Hence
\[
	\graphinvfrac{F}(G) + \graphinvfrac{F}(H) \le \frac{k}{d} + \frac{\ell}{d} \le \frac{n}{d} \le \graphinvfrac{F}(G + H) + \eps.
\]
Since $\eps$ was arbitrary, we must have $\graphinvfrac{F}(G) + \graphinvfrac{F}(H) \leq \graphinvfrac{F}(G + H)$.
\end{proof}

Next, we derive an alternative characterization of $\graphinvfrac{F}$, closely following the arguments of Bukh and Cox in the case of the fractional Haemers bound~\cite[Proposition~7]{bukh_cox}.

\begin{defn}
A \emph{rank-$r$-representation} of a graph $G$ with values in a linear-like semiring family $(F_n)$ is a homomorphism
\[
	\phi : G \to 2^{F_m}
\]
for some $m\in\N$ such that $\rk{\phi(v)} \ge r$ in $F_m$ for every $v\in V(G)$.
\end{defn}

\begin{prop}
\label{altetaf}
\[
	\graphinvfrac{F}(G) = \inf_{\phi, r} \, \left\{\: \frac{\rk{\bigcup_{v\in V(G)} \phi(v)}}{r} \: : \: \phi \textrm{ is a rank-}r\textrm{-representation of } G \:\right\}
\]
\end{prop}

\begin{proof}
We will show that every $\frac{n}{d}$ in the defining infimum \eqref{etaf} is dominated by some $\frac{\rk{\bigcup_{v\in V(G)} \phi(v)}}{r}$ and vice versa.

Given $G \to F_n/d$, we get $\phi$ as the composite $G \to F_n/d \to 2^{F_n}$, where $F_n/d \to 2^{F_n}$ is the canonical inclusion homomorphism. This lands in rank $d$ by definition, and trivially satisfies $\frac{\rk{\bigcup_{v\in V(G)} \phi(v)}}{d} \le \frac{n}{d}$.

Conversely, let $\phi : G \to 2^{F_m}$ be a rank-$r$-representation. We can assume without loss of generality that each $\phi(v)$ is a flat, since extending from $\phi(v)$ to $\phi(v)^{\perp\perp}$ does not decrease the rank and preserves the adjacency relations as well. Then we can restrict each $\phi(v)$ to an $r$-clique, which is guaranteed to exist by the definition of rank. Doing so results in a homomorphism $G \to F_m / r$, or equivalently $G\ltimes r \to F_m$. Since the rank of the image of this homomorphism is upper bounded by $n := \rk{\bigcup_{v\in V(G)} \phi(v)}$ by construction, it factors through $F_n$ by the linearity-like condition.
\end{proof}

\begin{prop}
	\label{lex_disj_mult}
$\graphinvfrac{F}$ is multiplicative on lexicographic and on disjunctive products.
\end{prop}

Here, the statement about the lexicographic product is due to Chris Cox, who observed that the following proof still goes through; the argument is again an adaptation of the one of Bukh and Cox~\cite{bukh_cox}.

\begin{proof}
Thanks to \cref{subinvs}, we only need to prove supermuliplicativity with respect to lexicographic product. So let $\graphinvfrac{F}(G \ltimes H) \ge \frac{n}{d} - \eps$, corresponding to some homomorphism $\phi : G \ltimes H \to F_n/d$. For fixed $v\in V(G)$, consider the associated homomorphism $\phi(v,-) : H\ltimes d \to F_n$, and put
\[
	r := \min_v \rk{\im{\phi(v,-)}}.
\]
Then $\graphinvfrac{F}(H) \le \frac{r}{d}$ by the linear-like property. Finally, consider the homomorphism
\[
	G \to 2^{F_n}, \qquad v \mapsto \im{\phi(v,-)}.
\]
It is a rank-$r$-representation by definition of $r$. \cref{altetaf} therefore gives $\graphinvfrac{F}(G) \leq \frac{n}{r}$. In total, we have
\[
	\graphinvfrac{F}(G) \, \graphinvfrac{F}(H) \le \frac{n}{r} \cdot \frac{d}{d} = \frac{n}{d} \le \graphinvfrac{F}(G\ltimes H) + \eps,
\]
which is enough.
\end{proof}

Comparing the three invariants, the inequality $\graphinvasymp{F}(G) \leq \graphinv{F}(G)$ is trivial, and likewise $\graphinvfrac{F}(G) \leq \graphinv{F}(G)$. The latter inequality can now be strenghtened:

\begin{cor}
For every finite graph $G$, we have $\graphinvfrac{F}(G) \leq \graphinvasymp{F}(G)$.
\end{cor}

\begin{proof}
Apply the trivial inequality $\graphinvfrac{F}(G) \leq \graphinv{F}(G)$ to the powers $G^{\ast n}$, use multiplicativity of $\graphinvfrac{F}$, and take the limit $n\to\infty$.
\end{proof}

We now summarize the results obtained so far into our main theorem, of which~\ref{etaf_spectral} is the main part.

\begin{thm}
	For any linear-like semiring family $F = (F_n)$, the three graph invariants $\graphinv{F}$, $\graphinvfrac{F}$ and $\graphinvasymp{F}$ have the following properties:
\begin{enumerate}
\item All are monotone under graph homomorphisms, subadditive under joins, and submultiplicative under both lexicographic and disjunctive products.
\item For $\graphinv{F}$, subadditivity holds with equality.
\item\label{etaf_spectral} For $\graphinvfrac{F}$, subadditivity and both forms of submultiplicativity hold with equality. In particular, $\graphinvfrac{F}$ \textbf{is semiring-homomorphic}.
\item For any finite graph $G$,
\[
	\graphinvfrac{F}(G) \le \graphinvasymp{F}(G) \le \graphinv{F}(G).
\]
\item On $K_n$, all three invariants take the value $n$.
\end{enumerate}
\label{mainthm}
\end{thm}

We do not know whether the inequality $\graphinvfrac{F} \leq \graphinvasymp{F}$ is strict, or whether it is necessarily an equality, perhaps under suitable additional conditions on the semiring family $(F_n)$.

Before getting to our examples, we state a comparison criterion for the semiring-homomorphic graph invariants induced by two linear-like semiring families:

\begin{prop}
Let $F = (F_n)$ and $F' = (F'_n)$ be linear-like semiring families such that there is $C \in\N$ together with a homomorphism $F_n \to F'_{Cn}$ for every $n\in\N$ which takes sets of rank $r$ to sets of rank at least $Cr$ for every $r\in\N$. Then we have
\[
	\graphinvfrac{F'}(G) \leq \graphinvfrac{F}(G)
\]
for every finite graph $G$.
\label{frac_compare}
\end{prop}

\begin{proof}
	We show that if $G \to F_n / d$, then also $G \to F'_{Cn} / Cd$, from which the claim follows. But this is because of $F_n / d \to F'_{Cn} / Cd$, which we show like this: the assumed homomorphisms $\phi_n : F_n \to F'_{Cn}$ take $d$-cliques $C_1,C_2\subseteq V(F_n)$ with $C_1 \sim C_2$ to $d$-cliques $\phi_n(C_1) \sim \phi_n(C_2)$, which implies $\phi_n(C_1)^{\perp\perp} \sim \phi_n(C_2)^{\perp\perp}$. Since these flats are of rank at least $Cd$, each of them must itself contain a $Cd$-clique. Choosing such $Cd$-cliques arbitrarily results in the desired graph homomorphism $F_n / d \to F'_{Cn} / Cd$.
\end{proof}

\section{Recovering all known semiring-homomorphic graph invariants}
\label{sec_recover}

We now illustrate how \cref{mainthm}\ref{etaf_spectral} lets us reconstruct all known examples of semiring-homomorphic graph invariants. The corresponding examples of linear-like semiring families underline their geometric character.

\begin{ex}
\label{frac_chrom}
The very simplest semiring family $(F_n)$ is $F_n = K_n$, and it is easy to see that it is indeed linear-like. In this case, we get the fractional chromatic number as $\graphinvfrac{F}$, and it is well-known that $\graphinvfrac{F} = \graphinvasymp{F}$. The properties now follow from \cref{mainthm} are well-known in this case; for example, the multiplicativity of the fractional chromatic number under lexicographic products is~\cite[Corollary~3.4.5]{SU}.
\end{ex}

\newcommand{\field}{\mathbb{F}}

\begin{ex}
\label{symm_ex}
Let $\field$ be a field with involution $\alpha \mapsto \bar{\alpha}$ such that the subfield of fixpoints of the involution is a Euclidean field. The paradigmatic examples are $\field = \R$ with the identity involution, or $\field = \C$ with complex conjugation.

These assumptions guarantee that every finite-dimensional Hilbert space over $\field$ has an orthonormal basis, where a Hilbert space is defined as a vector space equipped with a positive definite hermitian sesquilinear form. The existence of an orthonormal basis follows e.g.~using Gram--Schmidt orthogonalization, resulting in the well-known fact that a hermitian form over any field with involution can be diagonalized~\cite[p.~543]{jacobson}, and then normalizing by the relevant square roots. We use $\field^n$ with its standard hermitian inner product as a standard $n$-dimensional Hilbert space. We then define a graph $F_n$ with vertex set given by the projective space $V(F_n) := \mathbf{P}^{n-1}(\field)$, and we declare two vertices represented by vectors $\alpha,\beta\in \field^n\setminus\{0\}$ to be adjacent if and only if $\langle \alpha, \beta \rangle = 0$. Up to homomorphic equivalence, we can also work with the set of unit vectors with respect to orthogonality as adjacency, or even with all of $\field^n\setminus\{0\}$.

Next, we show that $(F_n)$ is a semiring family. Using the canonical inclusions $\field^n \to \field^n \oplus \field^m$ and $\field^m \to \field^n \oplus \field^m$ together with the canonical isomorphism $\field^n \oplus \field^m \cong \field^{n+m}$, we obtain the desired homomorphism $F_n + F_m \to F_{n+m}$ in the form of maps $\mathbf{P}^{n-1}(\field) + \mathbf{P}^{m-1}(\field) \to \mathbf{P}^{n + m - 1}(\field)$ which preserve orthogonality. Similarly, the tensor product map
\[
	\field^n \times \field^m \longrightarrow \field^n \otimes \field^m \cong \field^{nm}, \qquad (\alpha,\beta) \longmapsto \alpha \otimes \beta
\]
also has the property that if $\alpha \sim \alpha'$ or $\beta \sim \beta'$, then also $\alpha \otimes \beta \sim \alpha' \otimes \beta'$. This proves that $F_n \ast F_m \to F_{nm}$ as well. We thus have a semiring family of graphs.

Now for a set of vertices $S$, the flat $S^{\perp\perp}$ coincides with the linear span of $S$: it must contain the linear span since every vector orthogonal to $S$ is also orthogonal to every linear combination; and conversely, Gram--Schmidt orthogonalization shows that $S^{\perp}$ is a subspace whose dimension is $n - \rk{S}$, and is therefore complementary to the span $\mathrm{lin}_\field(S)$. So in this case, the linearity-like condition of \cref{ll} is straightforward to see.

Applying \cref{mainthm} therefore produces a semiring-homomorphic graph invariant for every field $\field$ satisfying the present assumptions.
% These are the \emph{symmetric fractional Haemers bounds} of Zuiddam. \todo{ref?}
For $\field = \C$, this specializes to the \emph{projective rank} introduced by Man{\v{c}}inska and Roberson~\cite[Section~6]{MR}, essentially by definition of the latter. For $\field = \R$, the resulting semiring-homomorphic invariant turns out to be the same~\cite[Section~6]{MR}, since \cref{frac_compare} implies bidirectional inequality: complexification induces graph homomorphisms $\mathbf{P}^{n-1}(\R) \to \mathbf{P}^{n-1}(\C)$ satisfying the hypotheses with $C = 1$; conversely, regarding $\C^n$ as a real vector space of twice the dimension (and taking the real part of the inner product) results in a graph homomorphism $\mathbf{P}^{n-1}(\C) \to \mathbf{P}^{2n-1}(\R)$ satisfying the hypotheses with $C = 2$. Thus using $\field = \R$ recovers projective rank as well. \cref{mainthm} in this case also recovers known properties, such as the multiplicativity~\cite[Theorem~27]{CMal}.

More generally, if $\field$ is a real closed field or the imaginary quadratic extension of a real closed field, then the resulting invariants coincide with those arising from $\field = \R$ or $\field = \C$, respectively, since the existence of a graph homomorphism $G\to \mathbf{P}^{n-1}(\field)$ is a sentence in the first-order logic of ordered fields, and all real closed field have the same first-order theory~\cite[Section~5.4]{model_theory}.
\end{ex}

The previous example involves inner product spaces for symmetric (or hermitian) inner products. 
It therefore seems natural to ask whether there could also be \emph{symplectic} versions of these semiring families which give semiring-homomorphic graph invariants. So far we have not been able to make this idea work; the main problem is that the tensor product of antisymmetric matrices is \emph{symmetric} rather than antisymmetric.

Nevertheless, also the Haemers bounds~\cite{haemers} and fractional Haemers bounds~\cite{bukh_cox} can be interpreted in terms of bilinear forms~\cite{peeters}:

\begin{ex}
\label{symp_ex}
Let $\field$ be any field, and suppose that $U$ and $W$ are vector spaces over $\field$ together with a bilinear form $\beta : U \times W \to \field$. Choosing bases for $U$ and $W$ represents $\beta$ by a matrix $M$ with entries $M_{ij} := \beta(u_i, w_j)$. Following Peeters~\cite[p.~423]{peeters}, we associate to this data a graph $O_\beta$ with vertex set
\[
	V(O_\beta) := \{\: (x,y) \in U \times W \mid \beta(x,y) = 1 \:\},
\]
and adjacency $(x,y) \sim (x',y')$ if and only if $\beta(x,y') = \beta(x',y) = 0$. The \emph{left kernel} $K\subseteq U$ is the subspace
\[
	K := \{\: x\in U \mid \beta(x,y) = 0 \:\: \forall y \in W \:\},
\]
and similarly the \emph{right kernel} of $\beta$ is the subspace of $W$ given by
\[
	L := \{\: y\in W \mid \beta(x,y) = 0 \:\: \forall x \in U \:\}.
\]
Choosing complementary subspaces $U'$ and $W'$ results in direct sum decompositions $U = K \oplus U'$ and $W = L \oplus W'$. We write $\beta' : U' \times W' \to \field$ for the restriction of $\beta$ to these complementary subspaces, and claim that there is a homomorphic equivalence $O_\beta \simeq O_{\beta'}$. As a homomorphism $O_{\beta'} \to O_\beta$, we can simply choose the inclusion map. For $O_\beta \to O_{\beta'}$, we use the projection maps $U \to U'$ and $W \to W'$, which commute with the bilinear form by the complementarity to kernels assumption, and therefore indeed $O_\beta \simeq O_{\beta'}$. Moreover, since $\beta'$ is nondegenerate by construction, we also must have $\dim{U'} = \dim{W'}$; by choosing suitable bases, we can assume $U' = W' = \field^n$, and that $\beta'$ is the standard inner product
\[
	\langle -,-\rangle_n : \field^n \times \field^n \to \field
\]
without loss of generality. In this way, we see that every $O_\beta$ is homomorphically equivalent to $O_{\langle -,-\rangle_n}$ for $n = \rk{\beta}$.

So for $n\in\N$, we consider the graph $F_n := O_{\langle -,-\rangle_n}$. We thus have
\[
	V(F_n) := \{\, (x,y) \in \field^n \times \field^n \mid \langle x,y\rangle = 1 \,\},
\]
with $(x,y) \sim (x',y')$ if and only if $\langle x,y'\rangle = \langle x',y\rangle = 0$. Then a $k$-clique consists of $x_1,\ldots,x_k\in \field^n$ and $y_1,\ldots,y_k\in \field^n$ such that $\langle x_i,y_j\rangle = \delta_{ij}$, and it follows that the $(x_i)$ and $(y_i)$ are each linearly independent. This implies $\omega(F_n) \leq n$, and $\omega(F_n) = n$ follows since the standard basis achieves the bound. In the situation of the graph $O_\beta$ above, we conclude $\omega(O_\beta) = \rk{\beta}$.

Using the obvious isomorphism $\field^n \oplus \field^m \cong \field^{n + m}$ gives a homomorphism $F_n \to F_{n + m}$ upon sending $(x,y) \in F_n$ to $(x \oplus 0, y \oplus 0)$, and similarly for $F_m \to F_{n + m}$ using the second component. These two homomorphisms assemble to $F_n + F_m \to F_{n + m}$.

Concerning compatibility with the disjunctive product, we similarly fix an isomorphism $\field^{n m} \cong \field^n \otimes \field^m$ induced by a bijection between the components, and consider the map
\[
	(x_1, y_1) \times (x_2, y_2) \longmapsto (x_1 \otimes x_2, y_1 \otimes y_2).
\]
The resulting compatibility between tensor product and canonical pairing shows that this indeed respects the edges of our graphs, resulting in a homomorphism $F_n \ast F_m \to F_{nm}$.

Concerning the linear-like property, we use the fact that every flat is the orthogonal complement of some collection of vertices $\{(x_i,y_i)\}_{i\in I}$ with some potentially infinite index set $I$. Then the orthogonal complement consists of all vertices $(x',y')$ with $x'\in U$ and $y'\in W$, where $U\subseteq \field^n$ is the subspace
\[
	U = \{\: x' \in \field^n \mid \langle x', y_i \rangle = 0 \:\;\forall i \:\},
\]
and similarly,
\[
	W = \{\: y' \in \field^n \mid \langle x_i, y' \rangle = 0 \:\;\forall i \:\}.
\]
As we saw above, the resulting graph is indeed homomorphically equivalent to $F_k$, where $k$ is the rank of the bilinear pairing between $U$ and $W$. This shows that the linear-like condition indeed holds, where the rank is the correct one due to the discussion of clique numbers above.

We now claim that the associated invariants $\graphinv{F}(G)$ and $\graphinvfrac{F}(G)$ coincide with the Haemers bound~\cite{haemers} and the fractional Haemers bound~\cite{blasiak,bukh_cox}\footnote{The fractional Haemers bound was first introduced as \emph{fractional minrank} in~\cite{blasiak}.}, respectively, of the complementary graph $\bar{G}$ (see also~\cite{peeters}). Using the notation of~\cite{bukh_cox}, the Haemers bound of $\bar{G}$ is the smallest rank of a matrix $M \in \field^{V(G) \times V(G)}$ with $M_{aa} = 1$ for all $a\in V(G)$ and $M_{ab} = 0$ for all adjacencies $a\sim b$ in $G$. Let $M$ be such a matrix achieving the smallest rank, for which we write $k$. Then $M$ defines a bilinear pairing $\beta : \field^{V(G)} \times \field^{V(G)} \to \field$ given by $(x,y)\mapsto x^t M y$. We have $O_{\beta} \simeq F_k$ per the above. Since $G \to O_{\beta}$ by the assumptions on $M$, we also conclude $G \to F_k$.

Conversely, suppose that we have a homomorphism $G \to F_k$, encoded in vertex labellings $x : G \to \field^k$ and $y : G \to \field^k$. Then the matrix $M_{vw} := \langle x_v,y_w \rangle$ has the required properties by construction. Hence our $\graphinv{F}(G)$ coincides with the Haemers bound of $\bar{G}$. Using~\cite[Proposition~6]{bukh_cox}, it now also follows that our $\graphinvfrac{F}(G)$ coincides with the fractional Haemers bound of $\bar{G}$.
\end{ex}

Intuitively, we can now also understand the (fractional) chromatic number of \cref{frac_chrom} in a new light: it plays the role of the (fractional) Haemers bound over the `field with one element'~\cite{Fun}.

\begin{ex}
\label{lovasz}
We now explain how the Lov\'asz number~\cite{lovasz} arises via our construction.

Let $\ell^2$ be the real Hilbert space of square integrable sequences $(x_i)_{i\in\N}$, and consider
\[
	F_n := \{ \: x \in \ell^2 \mid x_0 = 1, \: \|x\|^2 = n \:\},
\]
considered as a graph with orthogonality as adjacency. Equivalently, we can define $F_n$ by using the unit vectors $x$ in any separable Hilbert space which satisfy in addition $\langle c,x\rangle = n^{-1/2}$ for some fixed unit vector $c$.

We get $F_n + F_m \to F_{n + m}$ via, using the shorthand notation $x = (x_0, x_+)$,
\begin{align}
\begin{split}
\label{lovasz_add}
	F_n \to \R^2 \oplus \ell^2 \oplus \ell^2,\qquad & x \longmapsto \left(1,\, \sqrt{\frac{m}{n}},\, \sqrt{1+\frac{m}{n}}\cdot x_+,\, 0\right) \\
	F_m \to \R^2 \oplus \ell^2 \oplus \ell^2,\qquad & y \longmapsto \left(1,\, -\sqrt{\frac{n}{m}},\, 0,\, \sqrt{1+\frac{n}{m}} \cdot y_+\right)
\end{split}
\end{align}
and using some orthogonal isomorphism $\R^2 \oplus \ell^2 \oplus \ell^2 \cong \ell^2$ which takes the first component to the first component. In the above expressions, the components have been constructed precisely in such a way that the conditions for membership in $F_{n+m}$ are satisfied, such that the images of the two maps are elementwise orthogonal, and such that the orthogonality relations within $F_n$ and $F_m$ are preserved.

Showing $F_n \ast F_m \to F_{nm}$ is even easier: we similarly choose some identification $\ell^2 \otimes \ell^2 \cong \ell^2$ with $e_0\otimes e_0 \mapsto e_0$, and simply use the tensor product map $(x,y) \to x\otimes y$.

Concerning the linear-like condition, we first show $\omega(F_n) = n$. For $\omega(F_n) \geq n$, it is enough to consider the standard basis vectors $e_1,\ldots,e_n \in\ell^2$ and $c = n^{-1/2} \sum_{i=1}^n e_i$ in the alternative description mentioned above. We now show $\omega(F_n) \leq n$. Suppose that there were vectors $x^{(1)},\ldots,x^{(n+1)}\in F_n$ which are pairwise orthogonal. Then putting $y := n\, e_0 - \sum_{i=1}^{n+1} x^{(i)}$ results in a vector whose norm squared is given by
\begin{align*}
	\langle y, y\rangle & = n^2 \langle e_0, e_0 \rangle + \sum_{i=1}^{n+1} \langle x^{(i)}, x^{(i)}\rangle - 2 n \sum_{i=1}^{n+1} \langle e_0, x^{(i)} \rangle \\
	& = n^2 + (n+1) \cdot n - 2 \cdot (n+1) \cdot n = -n,
\end{align*}
which is absurd. Overall, we have proven $\omega(F_n) = n$.

It remains to be proven that every flat in $F_n$ is homomorphically equivalent to some $F_k$ for $k \leq n$, which is enough because of $\omega(F_k) = k$. As in \cref{symp_ex}, we consider a family of vertices $\{x^{(i)}\}_{i\in I}$ and show that its orthogonal complement is homomorphically equivalent to some $F_k$. That complement consists of all $y\in F_n$ with $\langle y,x^{(i)}\rangle = 0$ for all $i\in I$. We can assume without loss of generality that the $x^{(i)}$ are linearly independent. Then by induction\footnote{Our argument shows in particular that if there are infinitely many linearly independent $x^{(i)}$, then the complement is empty.}, it is enough to show that the orthogonal complement is isomorphic to $F_{n-1}$ in the case $|I| = 1$, in which case we write $x := x^{(1)}$. Applying an orthogonal transformation which preserves the first component, we can assume $x = (1,-\sqrt{n-1},0,\ldots)$ without loss of generality. Then the orthogonal complement coincides exactly with the image of the embedding
\[
	F_{n-1} \to F_n, \qquad (1,y_+) \mapsto \left(1,\,\sqrt{\frac{1}{n-1}},\,\sqrt{\frac{n}{n-1}}\cdot y_+\right)
\]
which is an instance of~\eqref{lovasz_add} above for $F_{n-1} + F_1 \to F_n$.

We now claim that already the resulting graph invariant $\graphinv{F}(G)$ is closely related to the Lov\'asz number, in that
\beq
\label{eta_lov}
	\graphinv{F}(G) = \lceil\vartheta(\bar{G})\rceil.
\eeq
This is because $F_n$ is isomorphic to the graph of unit vectors in $\ell^2$, where adjacency $x\sim y$ means that
\[
	\langle x, y \rangle = - \frac{1}{n-1}.
\]
One way to see this isomorphism is to map $x \mapsto (1,\sqrt{n-1}\cdot x)$. If we use this definition of $F_n$ together with $n\in\R_+$ being an arbitrary nonnegative number, then it is a known fact that $\vartheta(\bar{G})$ is the infimum over all $n\in\R_+$ with $G \to F_n$~\cite[Theorem~8.1]{KMS}. This proves~\eqref{eta_lov}.

Finally, we claim that both $\graphinvfrac{F}(G)$ and $\graphinvasymp{F}(G)$ coincide with $\vartheta(\bar{G})$. To this end, we prove the inequalities
\[
	\vartheta(\bar{G}) \leq \graphinvfrac{F}(G) \leq \graphinvasymp{F}(G) \leq \vartheta(\bar{G}).
\]
The first inequality follows from $\vartheta(\bar{G}) \leq \graphinvfrac{F}(G) + 1$ by~\eqref{eta_lov}, using as well the fact that $\vartheta(\overline{G\ltimes d}) = d\, \vartheta(\bar{G})$ (as an instance of~\cite[Corollary, \S{21}]{knuth}),
\[
	\vartheta(\bar{G}) = d^{-1} \vartheta(\overline{G\ltimes d}) \leq \frac{\graphinvfrac{F}(G\ltimes d) + 1}{d} = \graphinvfrac{F}(G) + d^{-1},
\]
which implies the claim in the limit $d\to\infty$. The third inequality $\graphinvasymp{F}(G) \leq \vartheta(\bar{G})$ follows directly from the fact that both sides preserve disjunctive powers $G\mapsto G^{\ast n}$. Now \cref{mainthm} recovers some standard properties of the Lov\'asz number of the complement, including some of its multiplicativity properties~\cite[\S{21}]{knuth}.
\end{ex}

In the case of the fractional chromatic number from \cref{frac_chrom}, it is known that the infimum $\graphinvfrac{F}(G) = \inf_d \frac{\graphinv{F}(G\ltimes d)}{d}$ is attained, making $\graphinvfrac{F}(G)$ into a rational number. Since the Lov\'asz number can be irrational, this statement does not generalize to arbitrary linear-like semiring families.

\section{Some open problems}
\label{sec_ops}

Clearly the most interesting open question is:

\begin{prob}
Does every semiring-homomorphic graph invariant $\FGraph\to\R_+$ arise from a linear-like semiring family? If not, are the ones which arise in this way at least dense in the topology of pointwise convergence in the set of functions $\FGraph\to\R_+$?
\label{all_ll}
\end{prob}

Another problem which may also be of interest from the perspective of combinatorial geometries is:

\begin{prob}
Find more examples of linear-like semiring families of graphs.
\end{prob}

The fact that the cases $\field = \C$ and $\field = \R$ in \cref{symm_ex} both define projective rank raises another nontrivial question: 

\begin{prob}
When to two given linear-like semiring families result in the same semiring-homomorphic graph invariants?
\end{prob}

Applying \cref{frac_compare} in both directions provides a sufficient criterion; the separation result for the fractional Haemers bounds of Bukh and Cox~\cite[Theorem~19]{bukh_cox} may provide some hints on how to find conditions which guarantee that two given linear-like semiring families induce different semiring-homomorphic graph invariants.

A closely related question is:

\begin{prob}
\label{fin_vs_infin}
Can the graph invariants reconstructed in \cref{symm_ex,symp_ex,lovasz} also be obtained from linear-like semiring families consisting only of \emph{finite} graphs?
\end{prob}

This is of interest since it would make these semiring families live themselves in the preordered semiring $\FGraph$. In light of \cref{finite_chi}, we could also replace $\FGraph$ by the preordered semiring of all graphs with finite chromatic number, and then we would want to know whether the above invariants can be obtained from linear-like semiring families consisting of graphs with finite chromatic number. We do not know whether the semiring families from \cref{symp_ex} have finite chromatic number (except in special cases, e.g.~for finite fields $\field$). But we do know what happens in the case of \cref{symm_ex} and \cref{lovasz}:

\begin{rem}
The semiring families $(F_n)$ of \cref{symm_ex} have finite chromatic number, as we show now. 

We first prove that if $x,y,z\in\field^n$ are unit vectors with $|\langle x,y\rangle| > 1/\sqrt{2}$ and $|\langle y,z\rangle| > 1/\sqrt{2}$, then also $|\langle x,z\rangle| > 0$. Rescaling by suitable phases implies that we can assume $\langle x,y\rangle > 1/\sqrt{2}$ and $\langle y,z\rangle > 1/\sqrt{2}$. Then the inner product of $w := x + z - \langle x + z, y\rangle y$ with itself evaluates to
\[
	0 \leq \langle w,w\rangle = 2 - (\langle x,y\rangle + \langle y,z\rangle)^2 + 2 \, \mathrm{Re}(\langle x,z\rangle) ,
\]
which indeed implies $|\langle x,z\rangle| \geq \mathrm{Re}(\langle x,z\rangle) > 0$. Hence to show that $F_n$ has finite chromatic number, it is enough to establish the existence of a finite set of unit vectors $b_1,\ldots,b_m \in \field^n\setminus\{0\}$ such that for every other unit vector $x\in\field^n\setminus\{0\}$, there is $i$ with $|\langle x,b_i\rangle| > 1/\sqrt{2}$. In the Euclidean case ($\alpha = \bar{\alpha}$ for all $\alpha\in\field$), we note that the components of every unit vector must lie in the subring $\field_{\mathrm{fin}}\subseteq \field$ of finite elements $\alpha\in\field$ characterized by $-n\leq \alpha \leq n$ for some $n\in\N$. Since $\field_{\mathrm{fin}}$ is totally ordered and contains $\Q$, every $\alpha\in\field_{\mathrm{fin}}$ has a \emph{standard part} $\mathrm{st}(\alpha)\in\R$, where $\mathrm{st} : \field_{\mathrm{fin}} \to \R$ is a ring homomorphism; $\mathrm{st}(\alpha)$ is defined to be the unique real contained in all rational intervals which contain $\alpha$. Since $\mathrm{st}$ is a ring homomorphism, it commutes with the standard inner product, $\mathrm{st}(\langle x,y\rangle) = \langle \mathrm{st}^{\times n}(x),\mathrm{st}^{\times n}(y)\rangle$. Hence it is enough to prove that there are finitely many unit vectors $b_1,\ldots,b_m\in\Q^n$ satisfying the above condition with respect to $x\in\R^n$; but this follows from compactness of the unit sphere and density of rational points on the unit sphere. In the non-Euclidean case, $\field$ is a degree two extension over its Euclidean subfield\footnote{There is $\alpha \neq \bar{\alpha}$. Replacing $\alpha$ by $\alpha - \bar{\alpha}$ shows that we can assume $\alpha$ to be purely imaginary, $\bar{\alpha} = -\alpha$, which means that $\alpha^2$ is in the Euclidean subfield; using the existence of square roots implies that we can assume $\alpha^2 = \pm 1$. Since the only square roots of $1$ are $\pm 1$, we must have $\alpha^2 = -1$. It is now easy to see that every element of $\field$ is a linear combination of $1$ and $\alpha$ over the Euclidean subfield.}, the claim follows in the same way using $b_1,\ldots,b_m\in\Q[i]^n$ satisfying the above condition with respect to all unit vectors $x\in\C^n$.
\end{rem}

The following observation and its proof were communicated to us by David Roberson.

\begin{rem}[Roberson] 
	In \cref{lovasz}, already the graph $F_3$ has infinite chromatic number, and therefore so do all $F_n$ with $n\geq 3$. The reason is that if $F_3$ had finite chromatic number $C\in\N$, then  by~\eqref{eta_lov}, every finite graph $G$ with $\vartheta(\bar{G}) \leq 3$ would satisfy $G\to F_3$, and therefore $\chi(G) \leq C$. But this is absurd, because the Kneser graphs $KG_{3r-1,r}$ satisfy $\vartheta(\overline{KG_{3r-1,r}}) = \frac{3r-1}{r} < 3$ by vertex-transitivity and $\vartheta(\overline{KG_{n,k}}) = |V(KG_{n,k})|\,\vartheta(KG_{n,k})^{-1} = \binom{n}{k} \binom{n-1}{k-1}^{-1} = \frac{n}{k}$~\cite[Theorems~8 and 13]{lovasz}. Taking $n = 3k-1$ therefore gives $\vartheta(\overline{KG_{3k-1,k}}) < 3$, but $\chi(KG_{3k-1,k}) = (3k-1) - 2k + 2 = k + 1$ by the Lov\'asz--Kneser theorem~\cite[Chapter~38]{thebook}, which is unbounded.
\end{rem}

Nevertheless, we do not know whether the Lov\'asz number can be reconstructed from a different linear-like semiring family involving only graphs which are finite or at least have finite chromatic number, so that \cref{fin_vs_infin} is still open.

Last but not least, it is very curious that all known semiring-homomorphic graph invariants are of a fractional nature. Before phrasing our concrete question on this observation, we need to briefly discuss fractional graph homomorphisms. We write $G\frachom H$ if there is a \emph{fractional graph homomorphism} from $G$ to $H$, by which we mean that there is $d\in\N$ such that $G \to (H\ltimes d) / d$, or equivalently $G\ltimes d \to H\ltimes d$. This does not coincide with the existing fractional homomorphism notion characterized by $\omega(G) \leq \omega_f(H)$~\cite[Theorem~7]{BM}, where $\omega_f$ is the fractional clique number, as we will see in \Cref{frac_diff}. For an example of a pair of graphs with a fractional homomorphism in our sense but no homomorphism, we have e.g.~$KG_{6,2} \not\to K_3$ but $KG_{6,2} \frachom K_3$ by \cref{monads}. We now have:

\begin{lem}
	Semiring-homomorphic graph invariants $\graphinvfrac{F}$ constructed via \cref{mainthm} are monotone under fractional graph homomorphisms in our sense.
	\label{frac_mon}
\end{lem}

\begin{proof}
	If $G\frachom H$, then $G\ltimes d \to H\ltimes d$ for some $d$, so that the claim follows from \cref{lex_disj_mult}.
\end{proof}

However, we do not know whether this stronger kind of monotonicity applies to all semiring-homomorphic graph invariants:

\begin{prob}
\label{graph_vs_fgraph}
Does $G \frachom H$ imply $\eta(G) \leq \eta(H)$ for every semiring-homomorphic graph invariant $\eta : \FGraph \to \R_+$?
\end{prob}

A negative answer would clearly imply a negative answer to both versions of \cref{all_ll}.

\begin{ex}
\label{frac_diff}
To see that our notion of fractional graph homomorphisms differs from the one of~\cite{BM}, it is enough to show that the semiring-homomorphic invariant $\chi_f$ is not monotone with respect to the latter notion of fractional homomorphisms. For example although $\omega(C_5) = 2 = \omega_f(K_2)$, we clearly have $\chi_f(C_5) = \frac{5}{2} \not\leq 2 = \chi_f(K_2)$.
\end{ex}

Our last open problem was communicated to us by David Roberson:

\begin{prob}[Roberson]
Is every semiring-homomorphic graph invariant multiplicative with respect to lexicographic product?
\end{prob}

As a special case, this would mean $\eta(G\ltimes d) = d \,\eta(G)$ for every semiring-homomorphic $\eta$, and therefore imply a positive answer to \cref{graph_vs_fgraph}, since $G\ltimes d\to H\ltimes d$ then yields $d\,\eta(G) \leq d\,\eta(H)$ and therefore $\eta(G) \leq \eta(H)$. A negative answer would imply a negative answer to both versions of \cref{all_ll} thanks to the multiplicativity with respect to lexicographic product of \cref{mainthm}, together with the fact that the multiplicativity equation is a closed condition in the topology of pointwise convergence.

\bibliographystyle{plain}
\bibliography{asymptotic_graphs}

\end{document}